\documentclass[12pt, А4]{article}

\usepackage{amssymb, amsfonts, amsmath}

\usepackage{amsthm}

\usepackage{mathrsfs}

\newtheorem{thm}{{\sc TТheorem}}

\newtheorem{cor}{{\sc Corrolary}}

\begin{document}

\title{The automorphisms of endomorphism semigroups of free Burnside groups}
\author{V.\,S.~Atabekyan}
%\address{Ереванский государственный университет}
%\email{avarujan@ysu.am}
%второй автор
%\author[I.\,I.~Ivanov]{И.\,И.~Иванов}
%\address{}
%\email{}

%\date{12.06.2011}
%\udk{512.54, 512.543, 512.544.43}

\maketitle

%\date{}
%\begin{fulltext}

\begin{abstract}
In this paper we describe the automorphism groups of the endomorphism semigroups of free Burnside groups $B(m,n)$ for odd exponents $n\ge1003$. We prove, that the groups $Aut(End(B(m,n)))$ and $Aut(B(m,n))$ are canonically isomorphic. In particular, if the groups $Aut(End(B(m,n)))$ and $Aut(End(B(k,n)))$ are isomorphic, then $m=k$.
\end{abstract}

%\begin{keywords}

\section{Введение}

The free Burnside group $B(m,n)$ is the free group of rank $m$ of the variety $\mathcal{B}_n$ of all groups satisfying the identity $x^n=1$. The group $B(m,n)$ is isomorphic to the quotient group of the absolutely free group $F_m$ of rank $m$ by normal subgroup $F_{m}^n$ generated by all $n$-th powers of its elements. It is well known (see \cite[Theorem 2.15]{Amon} that for all odd $n\ge665$ and rank $m> 1$ the group $B(m,n)$ is infinite (and even has exponential growth). According to one other theorem of S.I.Adyan (see \cite[Theorem 3.21]{Amon}) for $m> 1$ and odd periods $n\ge 665$ the center of $B(m,n)$ is trivial and hence, $B(m,n)$ is isomorphic to the inner automorphism subgroup $Inn(B(m,n))$ of the automorphism group $Aut(B(m,n))$. Other results on automorphisms and monomorphisms of the groups $B(m,n)$ appeared relatively recently in \cite{A09}-\cite{RC}. In this paper we describe the automorphism groups of the endomorphism semigroups of $B(m,n)$ for odd exponents $n\ge1003$. In particular, we prove, that the groups $Aut(End(B(m,n)))$ and $Aut(End(B(k,n)))$ are isomorphic if and only if $m=k$. This is a particular problem about $End(A)$, for $A$ a free algebra in a certain variety, was raised by B.I.Plotkin in \cite{P}. Analogous problems for $End(F)$ with $F$ a finitely generated free group or free monoid were solved by Formanek in \cite{F} and Mashevitzky and Schein in \cite{MS} respectively.

\medskip
For an arbitrary group $G$ consider a natural homomorphism $$\tau_G : Aut(End(G)) \to Aut(Aut(G))$$ taking each automorphism from the endomorphism semigroup of the group $G$ to its restriction on the subgroup of all invertible elements $Aut(G)$ of this semigroup. Obviously, any inner automorphism of $Aut(G)$ extends to an automorphism of the semigroup $End(G)$ in the natural way. Therefor, if all automorphisms of a group $Aut(G)$ are inner, then $\tau_G$ is surjective homomorphism. In particular, this is true for complete groups $Aut(G)$. Recall, that a group is called \textit{complete}, if its center is trivial and all its automorphisms are inner. More over, if $\phi:Inn(G)\to Inn(G)$ is an automorphism of the inner automorphism group of $G$ and $\phi(i_g)=i_{\alpha(g)},$ then it is not hard to check that  $\alpha:G\to G$ is an automorphism of $G$. By virtu of the relation $\alpha\circ i_g\circ\alpha^{-1}= i_{\alpha(g)}$ we get that the automorphism $\phi$ extends to the automorphism $i_\alpha$ of the semigroup $End(G)$. The subgroup $Inn(G)$ is characteristic in $Aut(G)$ for a complete groups $G$ by the criterion of Burnside. Hence, for any complete group $G$ the homomorphism $$\iota_G : Aut(End(G)) \to Aut(Inn(G))$$ taking each automorphism from $Aut(End(G))$ to its restriction on the subgroup $Inn(G)$ also is surjective.

\medskip
J.~Dyer and E.~Formanek in \cite{DF} proved that the automorphism group $Aut(F)$ is complete for any finitely generated non-abelian free group $F$. Therefore, we have an isomorphism $Aut(Aut(F))\simeq Aut(F)$. Later on Formanek in \cite{F} showed that the equality $Ker(\tau_F) = 1$ holds for the same free groups $F$. Thus, the groups $Aut(End(F))$ and $Aut(Aut(F))$ also are isomorphic and we have $Aut(End(F))\simeq Aut(Aut(F))\simeq Aut(F)$.

\medskip
In \cite{At13} we have proved that for any $m>1$ and odd $n\ge1003$ the inner automorphism group $Inn(B(m,n))$ is the unique normal subgroup of $Aut(B(m,n))$ among all its subgroups, which is isomorphic to a free Burnside group $B(s,n)$ of some rank $s\ge 1$. From this it follows that $Inn(B(m,n))$ is a characteristic subgroup in $Aut(B(m,n))$ and hence, the group $Aut(B(m,n))$ is complete. As was noted above, the restriction of every automorphism of the endomorphism semigroup $End(B(m,n))$ on the subgroup $Inn(B(m,n))$ induces an automorphism, because $Inn(B(m,n))$ is a characteristic subgroup in $Aut(B(m,n))$.

The aim of this paper is to prove the following

\begin{thm}\label{T}
Let us $\Phi$ and $\Psi$ be arbitrary automorphisms of the endomorphism semigroup $End(B(m,n))$ of the free Burnside group $B(m,n)$ of odd period $n\ge 1003$ and rank $m>1$. Then $\Phi=\Psi$ if and only if the restrictions of the automorphisms $\Phi$ and $\Psi$ on the subgroup $Inn(B(m,n))$ coincide, that is $$\Phi\big{|}_{Inn(B(m,n))}=\Psi\big{|}_{Inn(B(m,n))}.$$
\end{thm}

From Theorem \ref{T} immediately follows
\begin{cor}\label{c1}
The maps $$\tau_{B(m,n)} : Aut(End(B(m,n))) \to Aut(Aut(B(m,n)))$$
and $$\iota_{B(m,n)} : Aut(End(B(m,n))) \to Aut(Inn(B(m,n)))$$ are isomorphisms for any rank $m>1$ and odd period $n\ge 1003$.
\end{cor}

Taking into account that the groups $Aut(B(m,n))$ are complete for ranks $m>1$ and odd periods $n\ge 1003$ we get
\begin{cor}
For any automorphism $\Phi \in Aut(End(B(m,n)))$ there exists an automorphism $\alpha\in Aut(B(m,n))$ such that $\Phi(\varepsilon) = \alpha \circ \varepsilon\circ \alpha^{-1}$ for each endomorphism $\varepsilon\in End(B(m,n))$.
\end{cor}

\begin{cor}
For any odd $n\ge1003$ the groups $Aut(End(B(m,n)))$ and $Aut(End(B(k,n)))$ are isomorphic if and only if $m=k$.
\end{cor}
\begin{proof}
It follows from Corollary \ref{c1} that $Aut(End(B(m,n)))$ is isomorphic to $Aut(B(m,n))$. By Theorem 1.3 from \cite{At13} the automorphism groups $Aut(B(m, n))$ and $Aut(B(k, n))$ are
isomorphic if and only if $m = k$.
\end{proof}

\section{The proof of the main result}
Obviously, to prove Theorem \ref{T} it suffices to show that if the restriction of an automorphism $\Phi$ from $End(B(m,n))$ on the subgroup $Inn(B(m,n))$ is the identity automorphism, that is
\begin{equation}\label{e0}
\Phi\big{|}_{Inn(B(m,n))}=1_{Inn(B(m,n))},
\end{equation}
then $\Phi$ is the identity automorphism of the semigroup $End(B(m,n))$.

Suppose that the equality \eqref{e0} holds. We will prove that $$\Phi=1_{End(B(m,n))}$$ or equivalently we will prove that $\Phi(\varepsilon)=\varepsilon$ holds for each $\varepsilon\in End(B(m,n))$. More precisely we will show that the equality \begin{equation}\label{e}\varepsilon(a)^{-1}\cdot\Phi(\varepsilon)(a)=1\end{equation} holds for any $a\in B(m,n)$ and $\varepsilon\in End(B(m,n))$.

The inner automorphism of $B(m,n)$ induced by element $a\in B(m,n)$ is denoted by $i_a$. Consider an arbitrary endomorphism $\varepsilon\in End(B(m,n))$ and apply the product of endomorphisms $\varepsilon\circ i_a$ to an element $x\in B(m,n)$. By definition we have $$(\varepsilon\circ i_a)(x)=\varepsilon(i_a(x))=\varepsilon(a)\varepsilon(x)\varepsilon(a^{-1})=(i_{\varepsilon(a)}\circ\varepsilon)(x).$$
Hence, the equality
\begin{equation}\label{e1}
\varepsilon\circ i_a=i_{\varepsilon(a)}\circ\varepsilon
\end{equation}
holds.

To both sides of the equality \eqref{e1} applying the automorphism $\Phi$ and taking into account \eqref{e0} we get the equality
\begin{equation}\label{e2}
\Phi(\varepsilon)\circ i_a=i_{\varepsilon(a)}\circ\Phi(\varepsilon).
\end{equation}
Now the both sides of the equality \eqref{e2} applying to an arbitrary element $x\in B(m,n)$ we obtain
\begin{equation}\label{e3}
\Phi(\varepsilon)(a)\cdot\Phi(\varepsilon)(x)\cdot\Phi(\varepsilon)(a)^{-1}=\varepsilon(a)\cdot\Phi(\varepsilon)(x)\cdot \varepsilon(a)^{-1}.
\end{equation}
From the equality \eqref{e3} it follows that for any $a, x\in B(m,n)$ the element $\Phi(\varepsilon)(x)$ belongs to the centralizer of the element $\varepsilon(a)^{-1}\cdot\Phi(\varepsilon)(a)$.

By theorem of S.I.Adyan (see. \cite{Amon}, Theorem 3.2), the centralizer of each non-trivial element $b$ of the group $B(m,n)$ is a cyclic group of order $n$. Therefor, if there exist elements $x, y\in B(m,n)$ such that $\Phi(\varepsilon)(x)$ and $\Phi(\varepsilon)(y)$ don't belong in a same cyclic subgroup of  $B(m,n)$, then we have
\begin{equation}\label{e4}
\varepsilon(a)^{-1}\cdot\Phi(\varepsilon)(a)=1
\end{equation}
for any $a\in B(m,n)$. In particular, we get that $\Phi(\varepsilon)=\varepsilon$ for arbitrary automorphism $\varepsilon$, because $m>1$. Thus, to complete the proof of Theorem \ref{T} it remain to consider the case when $\operatorname{Im}(\Phi(\varepsilon))$ is a cyclic group.

\medskip
Suppose that $\operatorname{Im}(\Phi(\varepsilon))$ is a cyclic subgroup.

Let us $\{b_i\}_{i\in I}$ be a set of free generators for $B(m,n)$ and $|I|>1$. First we assume that $\varepsilon$ is the trivial endomorphism, i.e. $\varepsilon(b_i)=1$ for all $i\in I$ and verify that $\Phi(\varepsilon)=\varepsilon$. Suppose $\Phi(\varepsilon)$ is a non-trivial endomorphism. Then $\Phi(\varepsilon)(b_j)=b\not=1$ for some $j\in I$. Choose an automorphism $\alpha$ mapping the free generator $b_j$ to $b_j^{-1}$. Then, on the one hand we have $\varepsilon=\varepsilon\circ\sigma$ and hence, $\Phi(\varepsilon)=\Phi(\varepsilon\circ\sigma)=\Phi(\varepsilon)\circ\sigma$, because $\Phi(\sigma)=\sigma$. On the other hand we have $(\Phi(\varepsilon)\circ\sigma)(b_j)=b^{-1}\not=b=\Phi(\varepsilon)(b_j)$, because $b^2\not=1$ provided $n$ is odd. This contradiction shows that the image of the trivial endomorphism is the trivial endomorphism.

Now suppose that $\varepsilon(b_i)=a^{k_i}$ for some element $a\in B(m,n)$ and for each $i\in I$, where $k_i\in \mathbb{Z}$. We assume also that $\varepsilon(b_j)\not=1$ for some $j\in I$.

Let $d$ be a generator for the cyclic subgroup of additive group $\mathbb{Z}$ of integers generated by the set of the numbers $\{k_i\}_{i\in I}$. Obviously, applying some elementary transformations to the sequence of numbers $\{k_i\}_{i\in I}$ we can obtain a new sequence $\{s_i\}_{i\in I}$ such that $s_1=d$ and $s_i=0$ for $i\not=1$. Note that to any Nielsen transformation of system of generators $\{b_i\}_{i\in I}$ corresponds a Nielsen transformation of the system $\{\varepsilon(b_i)\}_{i\in I}$. The Nielsen transformations of the system $\{\varepsilon(b_i)\}_{i\in I}$ lead to the corresponding elementary transformations of the exponents $\{k_i\}_{i\in I}$ and vise versa. Consequently, there exist such Nielsen transformations of the system of free generators $\{b_i\}_{i\in I}$ which lied to the system of new free generators $\{y_i\}_{i\in I}$ of $B(m,n)$ satisfying to the conditions $\varepsilon(y_1)=a^d$ and $\varepsilon(y_i)=1$ for $i\not=1$.

\textbf{Case} 1. Let the period $n$ of the group $B(m,n)$ is a prim number. Consider the endomorphism $\alpha$ given by the relations $\alpha(y_1)=y_1$ and $\alpha(y_i)=1$ for $i\not=1$. Then $\Phi(\alpha)$ is a non-trivial endomorphism, because $\alpha$ is non-trivial. According to \eqref{e3} the element $\alpha(y_1)^{-1}\cdot\Phi(\alpha
)(y_1)$ belongs to the centralizer of $\Phi(\alpha
)(y_i)$ for all $i\in I$. The centralizer of $\Phi(\alpha
)(y_1)$ is a cyclic group of order $n$. Therefor, the elements $\alpha(y_1)^{-1}$ and $\Phi(\alpha
)(y_i)$ belongs to the centralizer of $\Phi(\alpha
)(y_1)$ for each $i$. Consequently, for any $i\in I$ there is an integer $t_i$ such that the equality $\Phi(\alpha)(y_i)=y_1^{t_i}$ holds. Evidently, the equality $\alpha\circ\alpha=\alpha$ also holds. Hence, $\Phi(\alpha)\circ\Phi(\alpha)=\Phi(\alpha)$. So, we have $t_it_1\equiv t_i( \operatorname{mod}n)$ for all $i\in I$. The integers $t_1$ and $t_1-1$ are relatively prime and $n$ is a prim number. Therefor, from $t_1^2\equiv t_1( \operatorname{mod}n)$ it follows that $t_1\equiv 0( \operatorname{mod}n)$ or $t_1\equiv 1( \operatorname{mod}n)$. Since $\Phi(\alpha)$ is a non-trivial endomorphism, we obtain that the congruence $t_1\equiv 1( \operatorname{mod}n)$ holds. Now for any $j\not=1$, $j\in I$ consider the Nielsen automorphism $\lambda_j$ given by equalities $\lambda_j(y_1)=y_1y_j$ and $\lambda_j(y_i)=y_i$ for $i\not=1$. It is easy to check that $\alpha\circ\lambda_j=\alpha$. Hence, we obtain $\Phi(\alpha\circ\lambda_j)=\Phi(\alpha)\circ\lambda_j=\Phi(\alpha)$ by virtu $\Phi(\lambda_j)=\lambda_j$. Therefor, $\Phi(\alpha)(\lambda_j(y_1))=\Phi(\alpha)(y_1)$, that is $y_1^{t_1+t_j}=y_1^{t_1}$ for $j\not=1$. This means that $t_j\equiv0( \operatorname{mod}n)$ for $j\not=1$. Thus, $\Phi(\alpha)(y_i)=y_1^{t_i}$, where $t_1\equiv 1( \operatorname{mod}n)$ and $t_i\equiv0( \operatorname{mod}n)$ for $i\not=1$. So, we get $\Phi(\alpha)=\alpha$.

Now let $b$ be an arbitrary non-commuting with $a$ element of the group $B(m,n)$ and $\gamma$ be an endomorphism of $B(m,n)$ given on the free generators by the formulae $\gamma(y_1)=a^d$ and $\gamma(y_i)=b$ for $i\not=1$. It is easy to verify that $\varepsilon=\gamma\circ\alpha$. Since $\operatorname{Im}(\gamma)$ is a non-cyclic subgroup, then $\Phi(\gamma)=\gamma$. Therefor, $\Phi(\varepsilon)=\Phi(\gamma)\circ\Phi(\alpha)=\gamma\circ\alpha=\varepsilon$. In the Case 1 the equality \eqref{e} proved.

\textbf{Case} 2. Let the period $n$ of $B(m,n)$ be a composite number and $n=n_1n_2$, where $1<n_1,n_2<n$. Consider the endomorphism $\delta_1 :B(m,n)\to B(m,n)$ given on the generators by the equalities $$\delta_1(y_1)=a^{d}\quad \mbox{and}\quad \delta_1(y_j)=y_k^{n_1}\quad \mbox{for}\quad j\not=1,$$ where $y_k$ is a fixed and non-commuting with $a^d$ generator of $B(m,n)$. Consider also the endomorphism $\delta_2 :B(m,n)\to B(m,n)$ defined by the equalities $$\delta_2(y_1)=y_1\quad \mbox{and}\quad\delta_2(y_j)=y_j^{n_2}\quad \mbox{for}\quad j\not=1.$$
Since the images of the endomorphisms $\delta_1$, $\delta_2$ are not cyclic, then $\Phi(\delta_i)=\delta_i$, $i=1, 2.$ From the definitions of endomorphisms $\delta_i$ immediately follows that $\delta_1\circ\delta_2=\varepsilon$. Therefor, we get $\Phi(\varepsilon)=\varepsilon.$ Theorem \ref{T} is proved.


\begin{thebibliography}{99}



\bibitem{Amon} S. I. Adian, {\it The Burnside Problem and Identities in Groups}, Ergebnisse der Mathematik und ihrer Grenzgebiete, {\bf95} (Springer-Verlag, Berlin, 1979).


\bibitem{A09} V.~S.~Atabekian, On subgroups of free Burnside groups of odd period $n\ge1003$. (in Russian) {\it Izv. Ross. Akad. Nauk Ser. Mat. } {\bf73}(5) (2009) 3--36; {\it Izv. Math.} {\bf73}(5) (2009) 861--892 (in English).

\bibitem{Ata09} V.~S.~Atabekyan, Monomorphisms of Free Burnside Groups, {\it Mat. Zametki} {\bf86}(4) (2009) 483--490 (in Russian); {\it Math. Notes} {\bf86}(4) (2009) 457--462 (in English).


\bibitem{Ai11} V.~S.~Atabekyan, Normal automorphisms of free Burnside groups. {\it Izv. RAN. Ser. Mat.} {\bf75}(2) (2011) 3--18 (in Russian); {\it Izv. Math.} {\bf75}(2) (2011) 223--237 (in English).



\bibitem{A13} V.~S.~Atabekyan, Splitting automorphisms of free Burnside groups. {\it Mat. Sb.} {\bf204}(2) (2013) 31--38 (in Russian); {\it Sbornik: Mathematics} {\bf204}(2) (2013) 182--189 (in English).


\bibitem{At13} V. S. Atabekyan, The groups of automorphisms are complete for free Burnside groups of odd exponents $n\ge 1003$, {\it Int. J. Algebra Comput.} {\bf23} (2013) 1485--1496.

\bibitem{At14} V.~S.~Atabekyan, Splitting automorphisms of order~$p^k$ of free Burnside groups are inner, {\it Mat. Zametki} {\bf95}(5) (2014) 651--655 (in Russian); {\it Mathematical Notes} {\bf95}(5) (2014) 586–589 (in English).


\bibitem{Ch}E. A. Cherepanov. Free semigroup in the group of automorphisms of the free Burnside group. {\it Comm. Algebra} {\bf33}(2) (2005) 539--547.

\bibitem{Ch1} E.~A.~Cherepanov, Normal automorphisms of free Burnside groups of large odd exponents, {\it Internat. J. Algebra Comput} \textbf{16}(5) (2006) 839--847.


\bibitem{RC}R.Coulon, Outer automorphisms of the free Burnside group, {\it Commentarii Mathematici Helvetici} {\bf88} (2013), 789--811.

\bibitem{DF} J.~Dyer, E.~Formanek, The automorphism group of a free group is complete. {\it J. London Math. Soc.} {\bf11}(2) (1975) 181--190.


\bibitem{F} E. Formanek, A question of B. Plotkin about the semigroup of endomorphisms of a free group, {\it Proc. Amer. Math. Soc.} {\bf130} (2002) 935--937.

\bibitem{MS} G. Mashevitzky, B. Schein, Automorphisms of the endomorphism semigroup of a free monoid or a free semigroup, {\it Proceedings of the American Mathematical Society} {\bf131}(6) (2003) 1655--1660.

\bibitem{P} B. I. Plotkin, {\it Seven Lectures on the Universal Algebraic Geometry}, Preprint, Institute of Mathematics, Hebrew University, (2000).






\end{thebibliography}
\end{document}